\documentclass{amsart}
\usepackage[T1]{fontenc}
\usepackage[latin9]{inputenc}
\usepackage{amsmath,amsthm,amsfonts,amscd,amssymb,eucal,latexsym,mathrsfs}
\usepackage{stmaryrd}
\usepackage{enumerate}
\usepackage{hyperref}

\usepackage{tikz,tikz-cd}
\usetikzlibrary{shapes.geometric}
\usepackage{tqft}

\usetikzlibrary{positioning}
\usepackage{float}
\usepackage{MnSymbol}
\usetikzlibrary{matrix}

\makeatletter
\renewcommand\part{\@startsection {part}{1}{\z@}%
                                   {-3.5ex \@plus -1ex \@minus -.2ex}%
                                   {2.3ex \@plus.2ex}%
                                   {\newpage\centering\normalfont\bfseries}}
\makeatother

\theoremstyle{plain}
\newtheorem{theorem}{Theorem}

\newtheorem{corollary}[theorem]{Corollary}
\newtheorem{lemma}[theorem]{Lemma}
\newtheorem{proposition}[theorem]{Proposition}

\theoremstyle{definition}
\newtheorem{remark}[theorem]{Remark}

\newtheorem{definition}[theorem]{Definition}

\newcommand{\acts}{\curvearrowright}
\newcommand{\G}{\Gamma}

\newcommand{\e}{\varepsilon}
\newcommand{\h}{\mathfrak{h}}

\newcommand{\IR}{\mathbb{R}}
\newcommand{\CI}{\mathbb{C}}

\newcommand{\ZI}{\mathbb{Z}}
\newcommand{\IN}{\mathbb{N}}

\newcommand{\M}{\mathrm{M}}

\DeclareMathOperator{\tr}{\mathrm{tr}}

\DeclareMathOperator{\Sym}{\mathrm{Sym}}
\DeclareMathOperator{\Aut}{\mathrm{Aut}}

\DeclareMathOperator{\impl}{\Rightarrow}
\DeclareMathOperator{\inj}{\hookrightarrow}

\DeclareMathOperator{\Map}{\mathrm{Map}}

\newcommand{\U}{\mathrm{U}}

\renewcommand{\Im}{\mathrm{Im}}

\DeclareMathOperator{\MS}{\mathrm{MS}}
\DeclareMathOperator{\CMS}{\mathrm{CMS}}

\newcommand{\sF}{\mathbf{F}}
\newcommand{\sG}{\mathbf{G}}
\newcommand{\sT}{\mathbf{T}}
\newcommand{\sD}{\mathbf{D}}

\newcommand{\sU}{\mathbf{U}}
\newcommand{\sB}{\mathcal{B}}
\newcommand{\sP}{\mathcal{P}}
\newcommand{\sQ}{\mathcal{Q}}
\newcommand{\AP}{\mathrm{AP}}

\renewcommand{\h}{h}

\newcommand{\sa}{\mathrm{sa}}
\renewcommand{\Re}{\mathrm{Re}}

\title{Universal entropy invariants}

\author{Robert Graham}
\author{Mikael Pichot}

\begin{document}

\begin{abstract}
We define  entropy invariants for abstract algebraic structures 
using an asymptotic Boltzmann formula.    
\end{abstract}

\maketitle

For an abstract algebraic structure $A$ (by which we mean a set with operations as understood in universal algebra)  we examine two (families of) numerical invariants associated with $A$, with values in $\{-\infty\}\cup [0,\infty]$, the entropy $h(A)$, and a related Minkowski dimension invariant, the entropy dimension $\dim(A)$.

The definitions rely on  the Boltzmann  formula 
``$h=k\log W$'' for computing the entropy
and  depend on the choice of spaces of microstates that approximate $A$ using simpler (usually finite or finite dimensional) asymptotic models $A_r$ for $A$.

Several concepts of entropy have been defined for structures such as groups,  dynamical systems, $C^*$-algebras, etc., using a Boltzmann type formula and which have obvious common patterns in their formulation. Many of these invariants  can be seen as specializations of  ``$\h$'' or ``$\dim$''.   In concrete situations, the invariants ``$\h$'' and ``$\dim$'' can be shown to coincide with known invariants such as the  Boltzmann/Shannon entropy, the Kolmogorov--Sinai entropy, Bowen's sofic entropy, the sofic dimension, Voiculescu's free entropy dimension, or various topological versions of entropy. 
We discuss this  further in  \textsection 4.   

The precise definition of $h(A)$ and $\dim(A)$  is given  in \textsection\ref{s-def universal} below.  
Roughly speaking, if we are given a tolerance of approximation $T$ for the  structure $A$ and $W_T$ denotes the ``size'' of the microstate space of $A$ at tolerance $T$, then the entropy is given by the standard  expression $k_T\log W_T$ where $k_T$ is a normalization constant (chosen appropriately as a function of  $T$). 
The invariants $h(A)$ and $\dim(A)$ are  obtained by a  limiting process as the tolerance $T$ shrinks
 (the limit over the  asymptotic models  $A_r$ being taken first). 

For the sake of applications, we will assume that  $A$ is endowed with an ``abstract functional analytic structure'', by which we mean that $A$ and $A_r$ have ``state symbols'' and  a metric with respect to which the operations and states are uniformly continuous. The latter assumption  is in line with the recent developments in continuous first order  logic (cf.\ \textsection \ref{s logic}) and it holds for the standard functional analytic structures.  The above approximation process defining entropy involves the states and the metric on $A$ and $A_r$.

The invariants $h(A)$ and $\dim(A)$ are local or ``finitary''  in the sense that they approximate the structure  $A$ on  finite subsets. These subsets are considered part of  the tolerance $T$ and are eventually limited over. The size $W_T$ of the microstate spaces is estimated using an $\e$ packing number, and $\e$ is also part of $T$.

The main result of the paper is the following ``Kolmogorov--Voiculescu invariance theorem''. It shows that $\h$ and $\dim$ satisfy a universal invariance property, and can be seen as a result in ``universal functional analysis''.

\begin{theorem}\label{th} The entropy $\h(A)$ is a topological invariant of $A$ and the  entropy dimension
$\dim(A)$  an algebraic invariant of $A$. 
\end{theorem}

In other words, if $A$ is algebraically generated by a subset $G$  then $\dim(A)=\dim(G)$, and if $A$ is topologically generated by $G$ (in the sense of Definition \ref{D- top gen}) then $\h(A)=\h(G)$. The precise statement  has an additional assumption on the normalization constants $k_T$ (see Theorem \ref{T-algebraic gen} and Theorem \ref{T - top independence theorem}). Even when $A$ is given as a specific model of a theory  there can still be a range of choices for the  microstates spaces which lead to  (a priori) distinct ways to measure the entropy in $A$.

Theorem \ref{th} is a further generalization of well--known results, starting with Kolmogorov's invariance of entropy under conjugacy for dynamical systems. But it also includes very recent results like the invariance under conjugacy of Bowen's sofic entropy or the invariance under isomorphism (orbit equivalence) of the sofic dimension for probability measure preserving equivalence relations. If $A$ is a finite von Neumann algebra then $\dim(A)$ is a technical modification of  Voiculescu's modified free entropy dimension $\delta_0$. The topological invariance of $\dim(A)$ is an open problem. This is expected to be hard if true as (just as for $\delta_0$) proving invariance would solve the free group factor isomorphism problem (using the Haagerup--Thorbj\o rnsen theorem, see \textsection\ref{S-applications}).  
The above theorem extends this type of ideas to fairly general algebraic structures, requiring that suitable asymptotic models exist. This makes the invariants $h$ and $\dim$ comparable in scope to  classical invariants such as the Betti numbers  in homological algebra.

 \setcounter{tocdepth}{1} 
\tableofcontents

\section{Definitions}\label{S-def}

\subsection{Macrostates}\label{s logic} We work within the framework of continuous logic as developed in the paper of Farah, Hart, and Sherman \cite{FHS}.

 Let $A$ be an algebraic 
structure with function symbols 
\[
\sF=\{f:A^{n}\rightarrow A,\ n\in\mathbb{Z}_{\geq0}\text{ the arity}\}
\]
(see e.g.\ \cite[Chapter 2]{BS}). Note that this contains constants.

Define the set $\sT$ of terms  to be the functions
obtained from finite iterations of elements of $\sF$ (if $f\in \sF$ is $n$-ary and $t_1,\ldots, t_n\in \sT$ are terms then $f(t_1,\ldots, t_n)$ is a term).
For $F\subset A$ and $T\subset\sT$ we define $F^{T}$
to be the subset of $A$ obtained after applications of elements of
$T$ on $F$.

We will also consider some functions 
\[
\sG=\{g:A^{n}\rightarrow \IR,\ n\in\mathbb{Z}_{\geq1}\text{ the arity}\}
\]
called the state symbols. (We could replace $\IR$ with any metric space $X$ serving as a ``screen'' for our state function $g$.)

\begin{remark} We assume as usual that the language is defined first and that $A$ is  an interpretation, and omit the superscript notation for simplicity. 
\end{remark}

In addition $A$ has a metric $d$ and a family $\sD$ of domains  exhausting $A$ which are complete and bounded with respect to $d$. The metric is assumed to be uniformly continuous with respect to the structure  in the sense that 
 every $D_{1},\ldots,D_{n}\in\sD$ and and $t\in\sT\cup\sG$ of arity $n>0$ there exists $\eta\colon(0,1]\rightarrow (0,1]$ such that for every $1\leq j\leq n$ every $(a_{1},\ldots,\hat{a_{j}},\ldots,a_{n})\in D_{1}\times\cdots\times \hat{D_{j}}\times\cdots\times D_{n}$ every $a,b\in D_{j}$ and every $\e>0$ we have 
\[
d(a,b)<\eta(\e) \impl d(t(a_{1}\ldots,a_{j-1},a,a_{j+1},a_{n}),t(a_{1},\ldots,a_{j-1},b,a_{j+1},\ldots,a_{n}))<\e.
\]
Note that $\eta$ depends on $D_1,\ldots, D_n$ and $t$.  

We call (uniform) continuity symbols of $A$ elements of the set $\sU:=\{$all functions $\eta\colon (0,1]\rightarrow (0,1]$ of the form 
\[
\eta:=\min_{t\in T}\eta^{D_1,\ldots, D_n}_t
\]
for $T\subset \sT\cup\sG\text{ finite and }D_1,\ldots, D_n\in \sD\}$.

\subsection{Microstates}

Suppose $B$ has the same structure as $A$. 
Let $F\subset A$ be a finite subset and $\delta\geq0$.
A map $\sigma:F\rightarrow B$ is a \emph{$\delta$-microstate} if the following two conditions are satisfied:

\begin{itemize} 
\item for
every $t\in\sT$ and $a_{1},\ldots,a_{n},a_{n+1}\in F$ such that $a_{n+1}=t(a_{1},\ldots,a_{n})$
then 
\[
d\left(\sigma(a_{n+1}),t(\sigma(a_{1}),\ldots,\sigma(a_{n}))\right)\leq \delta
\]
($\sigma$  is  a \emph{$\delta$-morphism}) 
\item  for every $g\in\sG$ and
$a_{1},\ldots,a_{n}\in F$ we have 
\[
d(g(a_{1},\ldots,a_{n}),g(\sigma(a_{1}),\ldots,\sigma(a_{n})))\leq\delta
\]
($\sigma$  is  \emph{$\delta$-state-preserving}) 
\end{itemize}
If $R\subset \sT\cup \sG$ is a subset and the  conditions hold for every $t\in R\cap \sT$ and every $g\in R\cap \sG$ then we say that $\sigma$ is a $(R,\delta)$-microstate. 

Furthermore,  a map  $\sigma:F\rightarrow B$ is 
\begin{itemize}
\item \emph{domain preserving} if $\sigma(a)$ is in any domain that $a$ is in for every $a\in F$. 
\item \emph{$\delta$-contractive} if
\[
d(\sigma(a),\sigma(b))\leq d(a,b)+\delta
\]
for every $a,b\in F$. Alternatively, we often  require that $d\in \sG$ resulting in the stronger inequality
\[
|d(a,b)-d(\sigma(a),\sigma(b))|\leq\delta
\]
for every $a,b\in F$ and microstate $\sigma$.
\end{itemize}

Given $F\subset A$, $R\subset \sT\cup \sG$ finite and $\delta\geq 0$, let
\begin{align*}
 & \MS(F,R,\delta,B)=\{\sigma\colon F\to B\ \ \text{domain preserving } (R,\delta)\text{-microstates}\}\\
 & \CMS(F,R,\delta,B)=\{\sigma\in\MS(F,R,\delta,B)\ \delta\text{-contractive}\}
\end{align*}
Note that $\CMS\subset \MS$ in general and $\MS= \CMS$ if $d\in \sG$. If $d\notin \sG$ then $\MS$ involves only  the metric  on $B$. These are really ``approximate microstates'' in that they only involve finite portions of the structure $A$.

Given $E\subset F$  define a metric on $\MS(F,R,\delta,B)$ and
$\CMS(F,R,\delta,B)$ 
\[
d_{E}(\sigma,\tau):=\sup_{s\in E}d(\sigma(s),\tau(s))
\]
and for $\e>0$ denote $N_{E,\e}(\ )$ the maximum cardinality of
an $\e$-separated set with respect to this metric and let $\tilde{N}_{E,\e}(\ )$
be the minimum cardinality of an $\e$-dense set.

\subsection{Entropy and dimension}\label{s-def universal}

Consider now a sequence $(A_{r})_{r\geq1}$ with the same structure
as $A$ and positive functions $N(r)$ and $L(\e)$ called the respectively the (inverse) Boltzmann and the packing normalization functions.

Given finite subsets $E\subset F$  of $A$ and $R\ll \sT\cup \sG$ define 
\begin{align*}
\h(E,F,R,\delta,\e) & :=\limsup_{r\to\infty}\frac{\log N_{E,\e}(\MS(F,R,\delta,A_{r}))}{N(r)}
\end{align*}
If $E'\subset E$, $F\subset F'$, $R\subset R'$, $\delta'\leq \delta$ and $\e\leq \e'$ then $\h(E',F',R',\delta',\e')\leq \h(E,F,R,\delta,\e)$.

If $E\subset B$ where $|E|<\infty$ and $B\subset A$ is an arbitrary subset then we let
\begin{align*}
\h(E,B, \e) & :=\inf_{\delta>0}\inf_{R\ll \sT\cup\sG}\inf_{E\subset F\ll B}\h(E,F,R,\delta,\e)
\end{align*}
where ``$F\ll B$'' means finite subset. Note that if $B$ is finite then 
\begin{align*}
\h(E,B, \e) & =\inf_{\delta>0}\inf_{R\ll \sT\cup\sG}\h(E,B,R,\delta,\e).
\end{align*}

The following definition is similar in spirit to that of Kolmogorov entropy (see \textsection \ref{S - KS entropy}):

\begin{definition}
The \emph{entropy} of a subset $G$ of $A$ is defined by
\[
h(G):=\sup_{E\ll  G}\sup_{\e>0}\h(E,G^\sT,\e).
\]
\end{definition}

Then we have a Minkowski dimension type invariant associated with $h$:

\begin{definition}
The \emph{entropy dimension}  of  $G\subset A$ is  defined by
\begin{align*}
\dim(G) & :=\sup_{E\ll G}\limsup_{\e\to 0}\frac{\h(E,G^\sT,\e)}{L(\e)}.
\end{align*}
\end{definition}

This is similar to Voiculescu's modified free entropy dimension $\delta_0$ for von Neumann algebras (cf.\ \textsection\ref{S-voi}).

These invariants are numbers in $\{-\infty\}\cup[0,\infty]$.
We say that $A$
is \emph{$A_r$-approximable} if $\h(A)\geq0$, and in this case we say that $A_r$ are  \emph{asymptotic models} for $A$.
 Most known applications are covered by two cases:
\begin{enumerate}
\item the \emph{finite case}: $A_r$ is finite as a set and then $L(\e)\equiv 1$, $h(G)=\dim (G)$. 
\item the \emph{finite dimensional case}: $A_r$ is a finite dimension real vector space endowed with a (suitably normalized) Euclidean metric and then the algebraic entropy dimension $\dim(G)$ is renormalized using a logarithm $L=|\log|$.
\end{enumerate}

\begin{remark}
Gromov has suggested the construction of general metric invariants  using strings of infima/suprema in the context of  first order logic (see \cite[\textsection 3 D$_+$]{Gro-metric book} for example). In \cite{Gro-entropy2012} it is explained  how the  classical Kolmogorov invariance theorem for dynamical entropy appears from a new angle through the eyes of category  theory (see \textsection \ref{S - KS entropy} for the universal approach on the same invariant).      
\end{remark}

We can characterize $A_{r}$-approximability using ultraproducts.
In the present context, see \cite{FHS}, ultraproducts have an additional condition on domains (so for example they coincide  with the usual notion for  von Neumann algebras).  Given an index set $I$ and an ultrafilter $\omega$
on $I$ define 
\[
{\prod_{r\in I}}^0 A_r:=\{a\in\prod_{r\in I} A_{r}\mid \exists D\in\sD,\ \{r:a_r\in D\}\in\omega\}
\]
and the ultraproduct
\[
\prod_{r\to \omega} A_r:={\prod_{r\in I}}^0 A_r\Big /\{d^\omega=0\}
\]
where $d^\omega$ is the pseudo metric on ${\prod}^0 A_r$ given by
\[
d^\omega(a,b):=\lim_{r\to\omega}d(a_r,b_r)
\]
(the limit is defined as the number $x$ with the property
that for all $\epsilon>0$ the set $\{r\in I:|d(a_r,b_r)-x|<\epsilon\}$ is in $\omega$)
and $\{d^\omega=0\}$ is the equivalence relation defined by  $a\sim b\iff d^{\omega}(a,b)=0$ so $d^\omega$ descends to a metric on the ultraproduct $\prod_{r\to\omega} A_r$. Each function or state symbol $f\in\sT\cup\sG$ is interpreted in the ultraproduct using the expression  $f((a_r)_{r\in I})=(f(a_r))_{r\in I}$.
Moreover each $D\in\sD$ corresponds to the quotient
set of $\{a\in {\prod}^0 A_r\mid \{r\in I:a_r\in D\}\in\omega\}$. It is readily
seen that the functions and states inherit the same moduli of continuity
as the $A_{r}$ so the ultraproduct as the same structure as the $A_r$. Furthermore, the \L o\'s theorem shows that $\prod_{r\to \omega} A_r$
inherits all the basic properties of the $A_{r}$ \cite[page 9]{FHS}, more precisely any the value of any formula in $\prod_{r\to \omega} A_r$
is the ultralimit of the formulas for the $A_{r}$.

\begin{proposition} 
$A$ is $A_{r}$-approximable if and only if $A$ has a state preserving
morphism into an ultraproduct of the form $\prod_{r\to \omega} A_r$.
\end{proposition}

The  morphism is injective if for example $d\in \sG$. 

\begin{proof}
Suppose that $A$ is $A_{r}$-approximable, by definition for any $F\ll A$
and $\delta>0$ we can find a $\delta$-microstate $\sigma_{(F,\delta)}\colon F\to A_{r}$ for some $r=r(F,\delta)$.
Let $I=\{(F,\delta):F\ll A,\ \delta=\frac{1}{n},\ n\in \IN\}$. There exists a non principal
ultrafilter $\omega$ on $I$ that contains all sets of the form
$\{(F',\delta')\in I:F\subset F'\ll A,\ \delta'<\delta\}$ for any
$F\ll A$ and $\delta>0$. Extending $\sigma_{(F,\delta)}$ arbitrarily to a map $A\to A_r$ the  product map
\[
A\ni  a\mapsto \prod_{(F,\delta)\in I}\sigma_{(F,\delta)}(a)\in {\prod}_{(F,\delta)\in I}^0 A_{r(F,\delta)}
\]
descend to a state preserving morphism $A\to \prod_{(F,\delta)\to \omega} A_{r(F,\delta)}$. The converse is clear by looking at the projection maps.
\end{proof}

\begin{remark}
In general the value of entropy may be altered when the $\limsup$ is replaced by a $\liminf$ in the definitions of $\h$ and $\dim$. A subset $G$ is said to be $\h$-regular if $\h(G)=\underline \h(G)$ where $\underline \h$ involves a $\liminf$ and $\dim$-regular if $\dim(G)=\underline \dim(G)$ where $\underline \dim$ involves a $\liminf$ twice. One can also take ultrafilters. Theorem \ref{T-algebraic gen} and \ref{T - top independence theorem} a valid when using a $\liminf$ in particular   the notion of regularity  depends on $A$ and not on the particular generating set $G$.   
\end{remark}

The following lemma records some simple but useful technical manipulations. From the metric/entropy  point of view these manipulations are trivial in the sense that they do not alter the value of $\e$. 

\begin{lemma} \label{L-techniques} 
\begin{enumerate}
\item If $B\subset A$ is arbitrary, $E,E'\ll B$ and $\e\geq 0$ then
\[
\h(E,B,\e)=\inf_{E\cup E'\subset F\ll B}\h(E,F,\e).
\]
\item If $B\subset A$ is arbitrary, $E,E'\ll B$ are finite, $\e\geq 0$,  and $B\subset B_0\subset B^\sT$ is arbitrary
\[
\h(E,B^\sT,\e)=\inf_{E\cup E'\subset F\ll B_0}\h(E,F^\sT,\e)
\]
\item 
 If $G\subset A$ is finite then 
\[
\dim(G)=\limsup_{\e\rightarrow0}\frac{\h(G,G^\sT,\e)}{L(\e)}.
\]
\end{enumerate}
\end{lemma} \begin{proof} 
\begin{enumerate}
\item ``$\leq$'' is clear, for the other direction note that for any
$E\subset F\ll B$ then $E\cup E'\subset F\cup E'\ll B$ and
since $E'\cup F\supset F$ we have 
\[
\h(E,F,\e)\geq \h(E,E'\cup F, \e)
\]
so ``$\geq$'' follows. 
\item ``$\leq $'' is clear since $F^T\ll B^\sT$ for every $F\ll B_0$ and every $T\ll \sT$.
Conversely let $\gamma>0$ and let $E\subset F\ll B^{\sT}$
 such that 
\[
\h(E,F, \e)\leq \h(E,B^\sT,\e)+\gamma
\]
Then there is $F_{0}\subset B$ and $T\subset\sT$ finite such that
$F\subset F_{0}^{T}$. We have $E'\cup E\subset E'\cup F_{0}^T$ and
\[
\h(E,E'\cup F_0^T,\e)\leq \h(E,F_{0}^T,\e)\leq\h(E,F,\e)
\]
so ``$\geq$'' follows letting $\gamma\to 0$. 
 
\end{enumerate}
\end{proof}

\section{The algebraic invariance of $\dim$}\label{S - top invariance}

The notation is as in \textsection\ref{S-def}. In this section we do not have that $d\in \sG$ a priori and we work with $\MS$.

\begin{theorem}\label{T-algebraic gen} Let $G_{1},G_{2}$ subsets
of $A$ such that $G_{1}^{\sT}=G_{2}^{\sT}$. For every approximation $(A_r)$ and every inverse Boltzmann normalization function $N$ we have
\[
\dim(G_{1})=\dim(G_{2})
\]
provided that the normalization function $L$ satisfies 
\[
\lim_{\e\to0}\frac{L({\eta(c\e)}/2)}{L(\e)}=1
\]
(convergence is assumed) for every function $\eta\in \sU$ and $c\in (0,1)$.
\end{theorem}

\begin{lemma} \label{L - alt invariance Voicu lemma} Assume that $G_{1}^{\sT}=G_{2}^{\sT}$ and let $E'\subset G_{1}$ be finite.
Then there are $E\subset G_{2}$, $c\in (0,1)$ and $\eta\in \sU$ such that for every $\e>0$ and  every $F$ finite such that $E\cup E'\subset F\subset G_{1}^{\sT}=G_{2}^{\sT}$
\[
\h(E',F^\sT,\e)\leq\h(E,F^\sT,{\eta(c\e)}/2).
\]
\end{lemma}

\begin{proof} Since $E'\subset G_{1}\subset G_2^\sT$ there exists $E=\{s_{1},\ldots,s_{n}\}\subset G_{2}$
and $T_{0}\subset\sT$ finite  such that for every $s\in E'$ there
exists $t\in T_{0}$ so that 
\[
s=t(s_{1},\ldots,s_{n}).
\]

Fix domains $D_i$ containing $s_i$ and define $\eta\in \sU$ by
\[
\eta:=\min_{t\in T_0} \eta_t^{D_{1},\ldots, D_{n}}\colon (0,1]\to (0,1].
\]

Let  $T_{0}\subset T\ll\sT$, $R\ll \sT\cup \sG$, 
$E\cup E'\subset F\ll G_{1}^{\sT}=G_{2}^{\sT}$
 and $r\in\mathbb{N}$.

Given $\e>0$, we let $\Omega_{\e}\subset\MS(F^T,R,\delta,A_{r})$ be a
maximal $\eta(\e)/2$ separated set with respect to  $d_{E}$. For any
$\sigma\in\MS(F^T,R,\delta,A_{r})$ there exists $\sigma'\in\Omega_{\e}$
so that for all $s_{i}\in E$ 
\[
d(\sigma(s_{i}),\sigma'(s_{i}))\leq\eta(\e).
\]
Since for $s\in E'$  
\[
d\left(\sigma(s),t(\sigma(s_{1}),\ldots,\sigma(s_{n}))\right)<\delta\text{ and }d\left(t(\sigma'(s_{1}),\ldots,\sigma'(s_{n})),\sigma'(s)\right)<\delta
\]
we have 
\begin{align*}
d(\sigma(s),\sigma'(s)) & \le2\delta+d\left(t(\sigma(s_{1}),\ldots,\sigma(s_{n})),t(\sigma'(s_{1}),\ldots,\sigma'(s_{n}))\right)\\
 & \le2\delta+d\left(t(\sigma(s_{1}),\ldots,\sigma(s_{n})),t(\sigma'(s_{1}),\sigma(s_{2}),\ldots,\sigma(s_{n}))\right)+\ldots\\
 & \hspace{1cm}+d\left(t(\sigma'(s_{1}),\ldots,\sigma'(s_{n-1}),\sigma(s_{n})),t(\sigma'(s_{1}),\ldots,\sigma'(s_{n}))\right)\\
 & \leq2\delta+n\e<(n+1)\e\text{ (wma \ensuremath{2\delta<\e})}
\end{align*}
where we use the fact that $\sigma(s_i),\sigma'(s_i)\in D_i$ since the microstates are domain preserving.

Therefore $\Omega_{\e}$ is $(n+1)\e$ dense set w.r.t.\ $d_{E'}$. Then
\begin{align*}
N_{E',2(n+1)\e}(\MS(F^T,R,\delta,A_{r}))\\
&\hspace{-2cm}\leq\tilde{N}_{E',(n+1)\e}(\MS(F^T,R,\delta,A_{r}))\leq|\Omega_{\e}|=N_{E,\eta(\e)/2}(\MS(F^T,R,\delta,A_{r})
\end{align*}
where $\tilde{N}$ denotes the minimal covering number. Dividing by $N(r)$ and taking a limit over $r$ and inf over
$T,R$ and $\delta$ we obtain
\[
\h(E',F^\sT, 2(n+1)\e)\leq\h(E,F^\sT,\eta(\e)/2).
\]
 \end{proof}

\begin{lemma}\label{L-alg inv eps lemma}
 Assume that $G_{1}^{\sT}=G_{2}^{\sT}$ and let $E'\subset G_{1}$ be finite.
Then there are $E\subset G_{2}$, $c\in (0,1)$ and $\eta\in \sU$ such that for every $\e>0$ we have
\[
\h(E',G_{1}^\sT,\e)\leq\h(E,G_{2}^\sT,{\eta(c\e)}/2)
\] 
\end{lemma}

\begin{proof} Let $E'\subset G_{1}$
be a finite set. By the previous lemma  there is a subset $E$ of $G_{2}$, $c\in (0,1)$ and $\eta\in \sU$ such
that for any $\e>0$ and any finite set $E\cup E'\subset F\subset G_{1}^{\sT}=G_{2}^{\sT}$
we have 
\[
\h(E',F^\sT,\e)\leq\h(E,F^\sT,{\eta(c\e)}/2)
\]
then 
\[
\inf_{E\cup E'\subset F\subset G_{1}^{\sT}}\h(E',F^\sT,\e)\leq\inf_{E\cup E'\subset F\subset G_{2}^{\sT}}\h(E,F^\sT,{\eta(c\e)}/2)
\]
then use Lemma \ref{L-techniques} (2). 
 \end{proof}
\begin{proof}[Proof of Theorem \ref{T-algebraic gen}] 
We prove
$\dim(G_{1})\leq\dim(G_{2})$. Let $E'\ll G_2$. By the previous lemma there are $E\ll G_1$  $\eta\in \sU$ and $c\in (0,1)$ such that for any $\e>0$ we have
\[
\h(E',G_{1}^\sT,\e)\leq\h(E,G_{2}^\sT,{\eta(c\e)}/2)
\] 
 Divide by $L(\e)$
\[
\frac{\h(E',G_{1}^\sT,\e)}{L(\e)}\leq\frac{\h(E,G_{2}^\sT,{\eta(c\e)}/2)}{L({\eta(c\e)}/2)}\frac{L({\eta(c\e)}/2)}{L(\e)}
\]
and take a lim sup as $\e\to 0$ to obtain by assumption on $L$ and since ${\eta(c\e)}/2\to 0$
\[
\limsup_{\e\to 0}\frac{\h(E',G_{1}^\sT,\e)}{L(\e)}\leq\limsup_{\e\to 0}\frac{\h(E,G_{2}^\sT,{\eta(c\e)}/2)}{L({\eta(c\e)}/2)}\leq \dim(G_2).
\]
The lemma follows taking a sup over $E'\ll G_1$.
\end{proof}

\begin{remark}\label{Rem - alg inv} We haven't assumed that  $G_i\subset D$ for some $D\in \sD$. If this is the case, then it is enough to suppose the microstates  $D$-preserving. 
\end{remark}

\section{The topological invariance of $\h$}

We now assume that our microstates are contractive (so the entropy is defined using $\CMS$) or  work with $\MS$ if $d\in \sG$.

\begin{definition}\label{D- top gen} We call a set $G\subset A$ is a \emph{topological generating set}
if for every $\e>0$ and $x\in A$ there exists $y\in G^{\sT}$
so that $d(x,y)<\e$ and $y$ is in any domain that $x$ is
in.\end{definition}

\begin{theorem}\label{T - top independence theorem} If $G_1$, $G_2$ are
two topological generating sets of $A$ then 
\[
\h(G_1)=\h(G_2).
\]
\end{theorem}

By symmetry we only have to show $\h(G_1)\leq\h(G_2)$. We may assume
that $\h(G_1)\neq\infty$ and $\h(G_2)\neq-\infty$. 

\begin{lemma} \label{L - Indep theorem lemma} Suppose $F,K\ll A$,  $S,T\ll\sT$ and $R\ll \sT\cup \sG$. For every $a\in  K^{S}$ choose a domain $D_a$ containing $a$ and let
\[
\eta:=\min_{t\in R}\min_{ a_1,\ldots a_n\in K^{S}}\eta_t^{D_{a_1},\ldots, D_{a_n}}.
\]
Note that $\eta$ depends on $R$, $K$ and $S$ but not on $F$ and $T$.

Let $\delta,\delta'>0$ and assume $\theta:K^{S}\rightarrow F^{T}$
is a domain preserving map such that $d(\theta(a),a)<\eta(\delta')$ for every $a\in K^{S}$ and $t(\theta(a_1),\ldots,\theta(a_n))\in F^T$ for all $t \in R$ and $a_1,\ldots, a_n\in K^{S}$ then the assignment
$\sigma\mapsto\sigma\circ\theta$ takes $\MS(F^T,R,\delta,A_{r})$ to
$\MS(K^S,R,2\delta+2\eta(\delta')+n_{R}\delta',A_{r})$ with $n_R$ being the maximal arity of $R$. 

\end{lemma}

\begin{proof} 
Let $\sigma\in\MS(F^T,R,\delta,A_{r})$ and take $t\in \sT\cap R$
of arity $n$. It is clear that $\sigma\circ\theta$ is domain preserving
from our setup.

If $a,a_{1},\ldots,a_{n}\in K^{S}$ satisfy $t(a_{1},\ldots,a_{n})=a$
then 
\[
d\left(\theta(a),t(\theta(a_{1}),\ldots,\theta(a_{n}))\right)\leq d(\theta(a),a)+d\left(a,t(\theta(a_{1}),\ldots,\theta(a_{n}))\right)\leq\eta(\delta')+n_{R}\delta'.
\]
Hence 
\begin{align*}
&d(\sigma\circ\theta(a),t(\sigma\circ\theta(a_{1}),\ldots,\sigma\circ\theta(a_{n})) \\
&\hspace{2cm}\leq d\left(\sigma\circ\theta(a),\sigma(t(\theta(a_{1}),\ldots,\theta(a_{n}))\right)\\
 &\hspace{4cm} +d\left(\sigma(t(\theta(a_{1}),\ldots,\theta(a_{n})),t(\sigma\circ\theta(a_{1}),\ldots,\sigma\circ\theta(a_{n})\right)\\
 & \hspace{2cm}\leq\eta(\delta')+n_{R}\delta'+\delta+\delta
\end{align*}
so $\sigma$ is  $(2\delta+\eta(\delta')+n_{R}\delta')$-multiplicative.

Next take $g\in \sG\cap R$ and $a_{1},\ldots,a_{n}\in K^{S}$ then 
\begin{align*}
&d\left(g(\sigma\circ\theta(a_{1}),\ldots,\sigma\circ\theta(a_{n})),g(a_{1},\ldots,a_{n})\right) \\
&\hspace{1cm} \leq d\left(g(\sigma\circ\theta(a_{1}),\ldots,\sigma\circ\theta(a_{n})),g(\theta(a_{1}),\ldots,\theta(a_{n}))\right)\\
 &\hspace{3cm} +d\left(g(\theta(a_{1}),\ldots,\theta(a_{n})),g(a_{1},\ldots,a_{n})\right)\\
 &\hspace{1cm} \leq\delta+n_{R}\delta'
\end{align*}
so $\sigma$ is  $(\delta+n_{R}\delta')$-multiplicative.

Finally for $a_{1},a_{2}\in K^{S}$ 
\begin{align*}
d(\sigma\circ\theta(a_{1}),\sigma\circ\theta(a_{2})) & \leq d(\theta(a_{1}),\theta(a_{2}))+\delta\leq d(a_{1},a_{2})+2\eta(\delta')+\delta
\end{align*}
so $\sigma$ is $(2\eta(\delta')+\delta)$-contractive.
\end{proof}

\begin{lemma}
If $G_1,G_2$ are topological generating sets then for every $\e>0$ and every $E\ll G_1$ there exists $H\ll G_2$, $c\in(0,1)$ and $\eta\in \sU$ such that 
\[
\h(E,G_1^\sT,\e)\leq\h(H,G_2^\sT,{\eta(c\e)}/2).
\]
\end{lemma}

\begin{proof}
Let $E\subset G_1$ and $\e>0$.
Since $G_2$ is topologically generating there exist finite sets $S_{0}\subset\sT$,
$H\subset G_2$ and a map $\rho:E\rightarrow H^{S_{0}}$ such that 
\[
d(\rho(a),a)<\e
\]
for all $a\in E$.

Let $R\ll \sT\cup \sG$, $S_{0}\subset S\ll\sT$ and $H\subset K\ll G_2$  and 
 choose any $\delta,\delta'>0$ satisfying $\delta<\e$ and  $2\eta(\delta')+n_{R}\delta'\leq\delta$ where $\eta$ is associated with $K^S$ and $R$ as in Lemma \ref{L - Indep theorem lemma}.

Since $G_1$ is topologically generating there exist finite sets $T\subset\sT$ and
$E\subset F\subset G_1$ and a map $\theta\colon K^{S}\to F^{T}$ such
that 
\[
d(\theta(a),a)<\eta(\delta')
\]
and $\theta(a)$ is in any domain $a$ is in for every $a\in K^{S}$. Furthermore by increasing $T$ if necessary we may assume that $t(\theta(a_1),\ldots,\theta(a_n))\in F^T$ for all $t \in R$ and $a_1,\ldots, a_n\in K^{S}$.

By Lemma \ref{L - Indep theorem lemma}, the assignment $\theta_{r}\colon\sigma\mapsto\sigma\circ\theta$
takes $\MS(F^T,R,\delta,A_{r})$ to 
\[
\MS(K^S,R,2\delta+2\eta(\delta')+n_{R}\delta',A_{r})\subset\MS(K^S,R,3\delta,A_{r})
\]
for every $r\in\mathbb{N}$. 

Let $\Omega_{r}\subset\MS(F^T,R,\delta,A_{r})$
be a maximal subset on which $\theta_r$ is injective and for which
$\theta_r(\Omega_{r})$ is $\e$-separated inside $\MS(K^S,R,3\delta,A_{r})$
with respect to $d_{H^{S_{0}}}$. We claim that $\Omega_{r}$ is $8\e$-dense in $\MS(F^T,R,\delta,A_{r})$.

Indeed let $\sigma_{0}\in\MS(F^T,R,\delta,A_{r})$, then by definition
of $\Omega_{r}$ we can find $\sigma\in\Omega_{r}$ such that 
\[
d(\sigma\circ\theta(a),\sigma_{0}\circ\theta(a))<2\e
\]
for every $a\in H^{S_{0}}$. Now for every $b\in E$  
\[
d\left(\theta\circ\rho(b),b\right)\leq d\left(\theta\circ\rho(b),\rho(b)\right)+d\left(\rho(b),b\right)<\eta(\delta')+\e<2\e
\]
so 
\begin{align*}
d\left(\sigma(b),\sigma_{0}(b)\right) & \leq d\left(\sigma(b),\sigma(\theta\circ\rho(b))\right)\\
 & \hspace{1cm}+d\left(\sigma(\theta\circ\rho(b)),\sigma_{0}(\theta\circ\rho(b))\right)+d\left(\sigma_{0}(\theta\circ\rho(b)),\sigma_{0}(b)\right)\\
 & \leq2(\delta+2\e)+2\e \leq8\e
\end{align*}
It follows that 
\begin{align*}
 & N_{E,16\e}(\MS(F^T,R,\delta,A_{r}))
  \leq\tilde{N}_{E,8\e}(\MS(F^T,R,\delta,A_{r}))
  \leq N_{H^{S_{0}},\e}(\MS(K^S,R,3\delta,A_{r}))
\end{align*}
Taking the $\log$, dividing by $N(r)$ and taking the limit on $r$,
we have
\[
\h(E,F^T,R,\delta,16\e)\leq\h(H^{S_{0}},K^S,R,3\delta,\e)
\]

Next we take the inf over $T\ll \sT$ and over $E\subset F\ll G_1^\sT$ followed by the inf over $\delta$,
over $S_0\subset S\ll \sT$, over $H\subset K\ll G_2^\sT$ and over $R\ll  \sT\cup \sG$, and  have
\[
\h(E,G_1^\sT,16\e)\leq\h(H^{S_{0}},G_2^\sT,\e).
\]
By Lemma \ref{L-alg inv eps lemma} since $H^{S_0}\subset G_2^{\sT}$ there exist $H'\subset G_2$,  $c\in (0,1)$ and $\eta'\in \sU$ such that for any $\e'>0$
\[
\h(H^{S_{0}},G_2^\sT,\e')\leq\h(H',G_2^\sT,{\eta'(c\e')}/2)
\] 
 so
\[
\h(E,G_1^\sT,16\e)\leq\h(H',G_2^\sT,{\eta'(c\e)}/2).
\]
\end{proof}

\begin{proof}[Proof of Theorem \ref{T - top independence theorem}]
If $\e>0$ and $E\ll G_1$ is finite then  there exists $H\subset G_2$ finite $c\in (0,1)$ and $\eta\in \sU$ such that 
\[
\h(E,G_1^\sT,\e)\leq\h(H,G_2^\sT,{\eta(c\e)}/2)\leq \sup_{H'\ll G_2}\sup_{\e'>0}\h(H',G_2^\sT,\e')= \h(G_2).
\]
The result follows taking a sup over $E$ and $\e>0$.
\end{proof} 

In fact a weak topological invariance result can be established for $\dim$. As with free entropy, one can introduce a relative version of the entropy dimension that is helpful.  Following up with the notation in \textsection \ref{S-def}   define the \emph{relative entropy} of $G\subset A$ \emph{in the presence of} $H\subset A$ by
\[
h(G:H):=\sup_{E\ll G}\sup_{\e>0}\h(E,(G\cup H)^{\sT},\e)
\]
and similarly the \emph{relative entropy dimension} 
\[
\dim(G:H):=\sup_{E\ll G}\limsup_{\e\to0}\frac{\h(E,(G\cup H)^{\sT},\e)}{L(\e)}.
\]
Then we have that:

\begin{proposition}\label{P - Voi relative} If $H_1,H_2\subset A$ are given with $H_2$ topologically generated by $G\cup H_1$ then  
\[
\dim(G: H_1)=\dim(G :H_1\cup H_2).
\]
\end{proposition}

(Compare \cite[\textsection 1]{Voicart}.) 

For example if $A=G''$ is a von Neumann algebra with the property $\G$ it contains asymptotically commuting projections  allowing to prove (see \cite{Voicart}) that $\dim(G)\leq 1$  where  the dimension ``$\dim$'' ($\leq \delta_0$) is defined in \textsection \ref{S-voi} below. Several ``vanishing results'' for $\dim$ can be established along these lines showing in particular cases that $\dim$ is a topological invariant. 

\section{Applications}\label{S-applications}

\subsection{Voiculescu's free probability theory}\label{S-voi}

Let $M$ be a von Neumann algebra with faithful normal tracial state $\tau$. 

Following Voiculescu, given self adjoint elements $X_1,\ldots, X_n\in M_\sa$ and noncommutative polynomials $P_1,\ldots, P_k\in \CI\langle X_1,\ldots, X_n\rangle$  set
\begin{align*}
\G(X_1,\ldots, X_n,P_1,\ldots, P_k,\delta,r)&:=\big\{(A_1,\ldots, A_n)\in (\M_r^\sa)^n\mid\\
 &\hspace{-2cm}|\tau(P_i(X_1,\ldots, X_n))-\tr(P_i(A_1,\ldots, A_n))|\leq \delta,\ \forall i=1\ldots k \}
\end{align*}
where $n,k,r\in \IN$, $\delta>0$ and $\tr$ is normalized trace on the $r\times r$ complex  matrices. 

Using the $\G(X_1,\ldots, X_n,P_1,\ldots, P_k,\delta,r)$ microstate spaces Voiculescu defines the free entropy $\chi(X_1,\ldots, X_n)$,  the free entropy dimension $\delta(X_1,\ldots, X_n)$ and  a technical modification $\delta_0(X_1,\ldots, X_n)$ of $\delta$ which turns out to be an algebraic invariant \cite{Voi}. 

In \cite{Jun1}  Jung proves that  $\delta_0(X_1,\ldots, X_n)$ can be computed as an asymptotic packing number
\[
\delta_0(X_1,\ldots, X_n)=\varlimsup_{\e\to 0}\inf_{\delta>0}\inf_{P_1,\ldots, P_k\in \CI\langle X_1,\ldots, X_n\rangle}\lim_{k\to\omega}\frac{N_\e(\G(X_1,\ldots, X_n,P_1,\ldots, P_k,\delta,r))}{r^2|\log\e|}.
\]
In our formulation we let:

\begin{itemize}
\item $\sF=\{0,1,$ addition, $x\mapsto -x$, multiplication, adjoint, scalar multiplication  $\lambda\cdot$ for any $\lambda\in \CI\}$ 
\item $\sG=\{\tau\}$ (or simply $\Re\tau$ and $\Im\tau$ for real valued tracial states)
\item $N(r)=r^2$ and $L(\e)=|\log \e|$
\end{itemize}
Since $\|x\|_2^2=\tau(xx^*)$ in $M$ we may equivalently include $\|\cdot\|_2$ in $\sG$.   

Assume that $\sD$ contains a set $D_t$ for every $t>0$ interpreted as the ball of radius $t$ in $M$ with respect to the uniform norm and let $d$ be the 2-distance on $M$ associated with $\|\ \|_2$. This gives a uniformly continuous algebraic structure on $M$. This is the framework in which Farah, Hart and Sherman show that one can axiomatize tracial von Neumann algebras \cite{FHS}. We  interpret $D_t$ as the uniform ball of radius $t$ in the algebra $\M_r$ of $r\times r$ matrices, which we also endow with the 2-distance then, by Theorem \ref{T-algebraic gen}, the resulting dimension ``$\dim$''  is an algebraic invariant:

\begin{corollary}
If $G\subset M$ is a subset and $A$ denotes the $*$-subalgebra of $M$ generated by $G$ then $\dim(A)=\dim(G)$.
\end{corollary}
There are minor differences between $\dim$ and $\delta_0$, including the fact that $G$ may be an arbitrary subset of $M$ (not necessarily self-adjoint) and that we do not assume that our microstates are mapping $G$ to self-adjoint matrices. If one starts from a self-adjoint tuple $G\subset M_\sa$ then the range of the microstates in $\MS$ will be approximately self-adjoint so taking real parts projects $G$ back to self-adjoints matrices and creates $\G$ microstates.   The main difference between $\dim$ and $\delta_0$ is the domain preserving condition.

It follows from the Haagerup--Thorbj\o rnsen theorem in \cite{HT} that if $X_1,\ldots, X_n$ is a free semicircular family then $\dim(X_1,\ldots, X_n)=n$ (by normalizing the microstates with respect to the uniform norm to obtain the domain preserving condition on $F\ll A$). Therefore, if ``$\dim$'' can be shown to be a topological invariant, then the free group factors $LF_n$ are not isomorphic. Note that $\h(X_1,\ldots, X_n)=\h(LF_n)=\infty$.   (One can also renormalize $h(E,B,\e)$ additively to recover Voiculescu's free entropy $\chi$.)

In \cite{HP} Hiai and Petz introduce a unitary version of Voiculescu's free entropy. This fits in our  framework with the obvious $\sF,\sG$ and $\sD$ containing a unique set interpreted as the unitary group itself,  $d$ being associated with the 2-norm. The same Boltzmann normalization $N(r)=r^2$ and packing normalization $L(\e)=|\log \e|$ apply. The associated  ``$\dim$'' coincides with the Hiai--Petz modified   free entropy unitary dimension $\delta_{0,u}(u_1,\ldots,u_n)$ for unitaries $u_1,\ldots, u_n\in M$.

\subsection{Kolmogorov--Sinai entropy and Bowen's sofic entropy}\label{S - KS entropy} 
Let $(X,\sB,\mu)$ be a standard probability space. We view finite  partitions  $\sP\subset \sB$ as measurable maps $\sP\colon X\to A$, where $A$ is a finite set, via $\sP_a:=\sP^{-1}(a)\subset X$.  Recall that the  usual Boltzmann  formula for  the entropy of $\sP$ 
 \[
H(\sP)=-\sum_{a\in A} \mu(\sP_a)\log \mu(\sP_a).
\]
This is established as follows. Let $\Map(\{1,\ldots, r\}, A)$ be the set of all maps $\{1,\ldots, r\}\to A$  and consider the microstate space (we follow the notation in \cite{Voisurv})
\[
\G(\sP, \delta, r):=\{\sQ \in \Map(\{1,\ldots, r\}, A),\ \sum_{a\in A} |\mu(\sP_a)-\mu_r(\sQ_a)|<\delta\}
\]
where $\{1,\ldots, r\}$ is endowed with the uniform probability measure $\mu_r$. Then
\[
H(\sP):=\inf_{\delta>0}\lim_{r\to\infty} \frac{\log |\G(\sP,\delta,r)|}{r}=-\sum_{a\in A} \mu(\sP_a)\log \mu(\sP_a).
\]
as 
\[
\log |\G(\sP,\delta,r)|\approx r\log r-\sum_{a\in A} r\mu(\sP_a)\log r \mu(\sP_a).
\]
From our point of view we have the $\sigma$-algebra $\sB$ with symbols  
\begin{itemize}
\item $\sF=\{0, 1, ^c,\cup,\cap\}$ where 0 is interpreted as $\emptyset$ and 1 as $X$
\item $\sG:=\{\mu\}$
\item $N(r)=r$ and $L(\e)=1$.
\end{itemize} 
The metric $d$ is interpreted as
\[
d(P,Q)=\mu(P\bigtriangleup Q)
\]
which defines a Polish metric on the measure algebra $L_\mu$ associated with $(X,\sB,\mu)$ (recall that $L_\mu$ is the space of measurable subsets of $X$ considered up to a.s.\ equality). The space is $X$ is approximated by  finite probability spaces $(\{1,\ldots, r\},\mu_r)$ with  $\mu_r$  the normalized counting measure  on $\{1,\ldots, r\}$, and the models for $L_\mu$ are the measure algebras $L_r= 2^{X_r}$ of $(X_r, 2^{X_r},\mu_r)$.

The entropy $h(G)$ of a (say, finite) subset $G\subset L_\mu$  defined using $\MS(G,T,\delta, L_r)$ as  in \textsection\ref{S-def} then satisfies
\[
h(G)=h(G^\sT)=h(\sP)=H(\sP)
\]
by the algebraic invariance theorem, where $\sP$ is the (finite) partition of $X$ generated by $G$ in $L_\mu$ and the last equality is an easy computation.

Lewis Bowen's sofic entropy for dynamical systems is  defined as follows, see \cite{Bowensofentropy}.  The join  $\sP\vee\sQ$ of two partitions $\sP\colon X\to A$ and $\sQ\colon X\to B$ is defined by $(\sP\vee\sQ)(x):=(\sP(x),\sQ(x))$ with values in $A\times B$. If $F\subset \Aut(X,\mu)$ is a finite set of probability measure preserving transformations then set
\[
\sP^F:=\bigvee_{s\in F} s(\sP).
\]
This is a finite partition  $X\to A^F$. If $\sP\colon X\to A$ is a finite partition, $F\subset \Aut(X,\mu)$ and $\sigma\colon F\to \Sym(r)$ is a map then the Bowen microstate space
\[
\AP_\sigma(\sP,F, \delta, r):=\{\sQ \in \Map(\{1,\ldots, r\}, A),\ \sum_{a\in A^F} |\mu(\sP^{F}_a)-\mu_r(\sQ^{\sigma(F)}_a)|<\delta\}
\]
 is the set of \emph{approximating partitions}. 

Let $\G$ be a sofic group (we review the definition in \textsection\ref{S-Sofic} below) with fixed sofic approximation $\Sigma=\{\sigma_i\colon \G\to \Sym(r_i)\}$ where  $\sigma_i$ is asymptotically multiplicative and asymptotically free on $F_i\subset G$ with $\bigcup_i^{\uparrow}F_i=\G$ (for definiteness  $\sigma_i$ is extended trivially to $\G$).

\begin{definition}[L. Bowen]
The sofic entropy of a finite partition $\sP\subset \sB$ is 
\[
\h_\mathrm{Bowen}(\sP)=\inf_{F\subset \G}\inf_{\delta>0}\varlimsup_{i\to\infty}\frac{\log |\AP_{\sigma_i}(\sP,F, \delta, r_i)|}{r_i}.
\]
(this depends on $\Sigma$ a priori).
\end{definition}

Bowen shows that his definition coincides with the classical Kolmogorov entropy for actions of  $\ZI$ and, more generally, for actions of amenable groups. He also
  extends it to infinite partitions with finite $H$ entropy.

From our point of view, consider again the measure algebra $L_\mu$ (with its Polish metric) and simply add $\G$ as unary function symbols:

\begin{itemize}
\item $\sF_\mathrm{dyn}:=\{0, 1, ^c,\cup,\cap\}\cup \G$ where the ``non dynamical'' $\sF$  defined above corresponds to $\G=\{e\}$.
\end{itemize}
The  function symbol $s\in \G$ is interpreted in $L_\mu$ as  the unary  operation $P\mapsto s(P)$ on $L_\mu$ for every $P\in L_\mu$.  The other symbols remain unchanged and so do their interpretations.

The asymptotic models for $L_\mu$, once a sofic approximation $\Sigma$ for $\G$ is chosen,  are the measure algebras  $L_{r_i}$ of  $(\{1,\ldots, r_i\},\mu_{r_i})$ with $s\in \G$ interpreted as the unary operation  $P\mapsto \sigma_i(s)(P)$. Using the microstate space $\MS(G,T,\delta ,L_{r_i})$ as in \textsection \ref{S-def} we have an entropy invariant $h(G)$ associated with an arbitrary subset $G\subset L_\mu$ (which depends on $\Sigma$). It is clear from the definition  that 
\[
h(\sP)=h_\mathrm{Bowen}(\sP)
\] 
when $\sP$ is a finite partition of $X$.

The following result \cite[Theorem 1.1]{Bowensofentropy} shows that $h_{\mathrm{Bowen}}$ is a invariant of conjugacy for group actions. This is a vast generalization of several earlier results including the classical Kolmogorov theorem which concerns actions of $\ZI$.

\begin{theorem}[Bowen] 
Let $\Sigma$ be a fixed sofic approximation of $\G$ and consider a pmp action $\G\acts (X,\sB,\mu)$.  Let $\sP,\sQ$ be dynamically generating partitions in $\sB$. Then $\h_\mathrm{Bowen}(\sP)=\h_\mathrm{Bowen}(\sQ)$.
\end{theorem}

Theorem \ref{T - top independence theorem} also establishes that $\h(\sP)=\h(\sQ)$ for arbitrary topologically generating subsets $\sP,\sQ\subset L_\mu$,  where the notion of ``topological generators'' of $L_\mu$ (in the sense of \textsection\ref{S - top invariance}) correspond to the classical notion of ``dynamical generators'' in $\sB$ given two  partitions $\sP,\sQ\subset \sB$.

\subsection{The sofic dimension and hyperlinearity}\label{S-Sofic} The notion of sofic group was introduced by Gromov and Weiss. In the terminology of \textsection \ref{S-def}:

\begin{definition}A sofic group is a group which is $\Sym(r)$-approximable (that is it admits the symmetric groups $\Sym(r)$ as asymptotic models). 
\end{definition}

Here $\sF=\{1,*\}$ and $\sG=\{\tau\}$ is interpreted as the (von Neumann) tracial state on the group $\G$ 
\[
\tau(s)=\begin{cases} 1 \text{ if } s=e\\ 0\text{ otherwise} \end{cases}
\]   
It is interpreted  as the normalized number of fixed points on $\Sym(r)$
\[
\tau(\sigma)=\mu_r\{x,\ \sigma(x)=x\}
\] 
 $\mu_r$ denotes the normalized counting measure  on $\{1,\ldots, r\}$.

The group $\G$ is has the discrete metric and $\Sym(r)$ has the Hamming metric  
\[
d_r(\sigma,\sigma')=\mu_r\{x,\ \sigma(x)\neq \sigma'(x)\}=1-\tau(\sigma\sigma'^{-1}),
\]
 which is  also often called the uniform metric when it is  viewed as a metric on $\Aut(\{1,\ldots, r\},\mu)=\Sym(r)$. There is one domain $\sD=\{\G\}$. 

The associated entropy invariant $h(G)$ with normalization 

\begin{itemize}
\item $N(r)=r\log r$ and $L(\e)\equiv 1$
\end{itemize}
coincide by definition with the ``sofic dimension'' $s(G)$ of $\G$  (see \cite{DKP1,DKP2}). The definition can be extended to pmp measure equivalence relations and pmp groupoids. Theorem \ref{T - top independence theorem} then implies \cite[Theorem 4.1]{DKP1} and \cite[Theorem 2.11]{DKP2}.

For example, if $F_n$ denotes the free group on $n= 1,2,\ldots, \infty$ letters then the sofic dimension satisfies $s(F_n)=n$ while if $\G$ is infinite amenable or a direct product of infinite groups then $s(\G)=1$. 

\begin{remark} If $\G\acts X$ is a free pmp action  of a countable group $\G$, then Bowen's sofic entropy is a conjugacy invariant for the action $\G\acts X$ for every sofic approximation of $\G$ (namely, this is an invariant of the space $X$ itself relative to the group  $\G$), while the sofic dimension will rather concern the orbit structure of the action $\G\acts X$ namely the orbit equivalence relation of the action
\[
R=\G\ltimes X=\bigcup_{s\in \G,\ x\in X}(x,sx)\subset X\times X
\] 
In the latter case ``topological invariance'' (which combines both the topology and the algebra structure/multiplication $(x,y)(y,z)=(x,z)$ on the orbit structure $R$ itself) creates an invariant  which ``ignores'' the group $\G$. 
\end{remark}

Similarly, a group $\G$ is said to be hyperlinear if it is $U(r)$-approximable, writing $U(r)$ for the unitary group (equivalently the group von Neumann algebra is embeddable in the ultrapower $R^\omega$ of the hyperfinite $\mathrm{II}_1$ factor $R$). The symbols $\sF$ and $\sG$ are the same and $\tau$ is interpreted in $U(r)$ as the normalized trace. Note that the restriction of $\tau$ to $\Sym(r)\subset U(r)$ (permutation matrices) coincides with the sofic interpretation of $\tau$. The associated entropy dimension invariant $\dim$ has normalization  
\begin{itemize}
\item $N(r)=r^2$ and $L(\e)= |\log \e|$
\end{itemize}
and it coincides with the Hiai-Petz invariant $\delta_{0,u}$. In this case $\dim$ is only known to be an algebraic invariant.

One can define relative versions of Voiculescu's free entropy dimension in the same way one recovers Bowen's entropy from the Boltzman entropy of $(X,\sB,\mu)$ by adding group element symbols rending the situation equivariant. 
Namely, let $M$ be a von Neumann algebra with tracial state $\tau$ and let $\G$ be a subgroup of $\Aut(M)$. Assume that $\G$ is hyperlinear and fix a $\U(r_k)$-approximation $\Sigma:=\{\sigma_k\colon \G\to \U(r_k) \}$. The relative microstate spaces $\MS(F,R,\delta,\M_{r_k})$ are defined as the microstate spaces for the algebraic structure with function symbols $\sF_\mathrm{dyn}:=\sF\cup \G$ where every $s\in \G$ acts as a unary function by automorphism on $M$ and by conjugacy by $\sigma_k(s)$ on $\M_{r_k}$. The corresponding entropy dimensions $\dim_\G(G)$ for $G\subset M$ coincides with $\dim(G)$ as defined in \textsection \ref{S-voi} for $\G=\{e\}$. The normalization is again $N(r_k)=r_k^2$ and $L(\e):=|\log\e|$. 

If $M=L^\infty(X)$ then one will preferably choose $\CI^r$ (with the usual Hermitian structure) as asymptotic models, and for every hyperlinear group $\G$ one can define an entropy invariant  using  $\MS(F,T,\delta,\CI^{r_k})$ associated with $\Sigma=\{\sigma_k\colon \G\to \U(r_k) \}$. In this case $s\in \G\subset \sF_\mathrm{dyn}$ is interpreted using the natural action of $\G$ on $L^\infty(X)$ and the standard action of $\U(r_k)$ on $\CI^{r_k}$. This now uses the Boltzmann normalization $N(r_k)=r_k$ and trivial packing normalization $L(\e):=1$. If $\sigma_k\colon \G\to\Sym(r_k)\inj \U(r_k)$
factorizes through a sofic approximation of $\G$ then it is easy to see that this ``hyperlinear entropy'' invariant (also a topological invariant) coincides with Bowen's sofic entropy.

\subsection{Topological entropy}

In \cite{Voitop} Voiculescu considers a topological version of free entropy for $C^*$-algebras using norm microstates.  In this case the $C^*$-norm only is involved and the resulting ``$\dim$'' defined using 
\begin{itemize} 
\item $\sF=\{0,1,$ addition, $x\mapsto -x$, multiplication, adjoint, scalar multiplication  $\lambda\cdot$ for any $\lambda\in \CI\}$ 
\item $\sG=\{d\}$
\item $N(r)=r^2$ and $L(\e)= |\log\e|$.
\end{itemize}
coincides  with Voiculescu's $\delta_\mathrm{top}$. 

In \cite{KL} Kerr and Li have introduced a topological version of (Lewis) Bowen's sofic entropy and established the variational principle in that generality.  Their invariant generalizes  the classical (``orbit dispersion'') definition of (Rufus) Bowen and Dinaburg. It is easy to see that this fits in our framework as well. Namely, if $\G\acts X$ is continuous with $X$ compact, the associated algebraic structure $A$ is simply $C(X)$ viewed as a $*$-algebra with metric $d$ interpreted as the $C^*$-norm:
\begin{itemize} 
\item $\sF_\mathrm{dyn}=\{0,1,$ addition, $x\mapsto -x$, multiplication, adjoint, scalar multiplication  $\lambda\cdot$ for any $\lambda\in \CI\}\cup \G$ 
\item $\sG=\{d\}$
\item $N(r)=r$ and $L(\e)\equiv 1$.
\end{itemize}
where the group elements are interpreted through the sofic approximation as in Bowen's sofic entropy.  One can also define a relative notion of entropy for Voiculescu's topological free entropy replacing $\sF$ with  $\sF_{\mathrm{dyn}}$  using  $N(r)=r^2$ and $L(\e)= |\log\e|$.

\end{document}